\documentclass[letterpaper, 11pt]{article}
\usepackage{graphicx}%
\usepackage{multirow}%
\usepackage{amsmath,amssymb,amsfonts}%
\usepackage{amsthm}%
\usepackage{mathrsfs}%
\usepackage[title]{appendix}%
\usepackage{xcolor}%
\usepackage{textcomp}%
\usepackage{manyfoot}%
\usepackage{booktabs}%
\usepackage{algorithm}%
\usepackage{algorithmicx}%
\usepackage{algpseudocode}%
\usepackage{listings}%
\usepackage[left=1in, right=1in, top=1in, bottom=1in]{geometry}
\usepackage{authblk}

\newtheorem{theorem}{Theorem}
\newtheorem{proposition}[theorem]{Proposition}%

\theoremstyle{thmstyletwo}%
\newtheorem{lemma}{Lemma}%

\theoremstyle{thmstylethree}%
\newtheorem{definition}{Definition}%

\begin{document}

\title{Branching with a pre-specified finite list of $k$-sparse split sets for binary MILPs}


\author[1]{Santanu S. Dey \thanks{santanu.dey@isye.gatech.edu}}

\author[2]{Diego Mor\'an\thanks{morand@rpi.edu}}

\author[1]{Jingye Xu \thanks{jxu673@gatech.edu}}

\affil[1]{{School of Industrial and Systems Engineering}, {Georgia Institute of Technology}}

\affil[2]{{Department of Industrial and Systems Engineering}, {Rensselaer Polytechnic Institute}}

\maketitle

\abstract{When branching for binary mixed integer linear programs with disjunctions of sparsity level $2$, we observe that there exists a finite list of $2$-sparse disjunctions, such that any other $2$-sparse disjunction is dominated by one disjunction in this finite list. For sparsity level greater than $2$, we show that a finite list of disjunctions with this property cannot exist. This 
leads to the definition of covering number for a list of splits disjunctions. Given a finite list of split sets $\mathcal{F}$ of $k$-sparsity, and a given $k$-sparse split set $S$, let $\mathcal{F}(S)$ be the minimum number of split sets from the list $\mathcal{F}$, whose union contains $S \cap [0, \ 1]^n$. Let the covering number of $\mathcal{F}$ be the maximum value of $\mathcal{F}(S)$ over all $k$-sparse split sets $S$. We show that the covering number for any finite list of $k$-sparse split sets is at least $\lfloor k/2\rfloor $ for $k \geq 4$. We also show that the covering number of the family of $k$-sparse split sets with coefficients in $\{-1, 0, 1\}$ is upper bounded by $k-1$ for $k \leq 4$.}

\section{Introduction}\label{sec:intro}

Land and Doig~\cite{land1960automatic} invented the branch-and-bound procedure to solve  mixed integer linear programs (MILP). Today, all state-of-the-art MILP solvers use the branch-and-bound procedure at its core.  
An important decision in formalizing a branch-and-bound algorithm is to decide the method to partition the feasible region of the linear program corresponding to a node in the branch-and-bound tree.  Given $\pi \in \mathbb{Z}^n$ and $\eta \in \mathbb{Z}$, a general way to partition a feasible region where all variables are binary is to the use the following disjunction for $x \in \{0, 1\}^n$:
$\left(\pi^{\top}x \leq \eta \right) \vee \left(\pi^{\top}x \geq \eta + 1 \right),$
in order to create two child nodes. The open set $$S(\pi, \eta):= \left\{x \in \mathbb{R}^n \,|\, \eta < \pi^{\top}x  < \eta+ 1\right\},$$ is called \emph{split set} and the associated disjunction is called \emph{split disjunction}. We say a split set $S(\pi, \eta)$ is $k$-sparse if the number of non-zero entries of $\pi$, denoted by $\|\pi\|_0$, is at most $k$, that is, $\|\pi\|_0 \leq k.$

Most state-of-the-art MILP solvers are based on branch-and-bound trees built using $1$-sparse split disjunctions; such branch-and-bound trees are called simple branch-and-bound trees~\cite{dey2021lower}. One rationale for using $1$-sparse split disjunctions is to maintain the sparsity of linear programs solved at child nodes; see discussion in~\cite{dey2015approximating,dey2018analysis}. Recently,~\cite{dey2021branch,borst2020integrality} showed that on random instances, using $1$-sparse split disjunctions is sufficient to obtain a polynomial size branch-and-bound tree when the number of constraints are fixed.  However, several papers have shown the power of constructing branch-and-bound trees with dense disjunctions. See, for example, the papers~\cite{owen2001experimental,aardal2004hard,mahajan2009experiments,mehrotra2011branching,karamanov2011branching,cornuejols2011improved,mahmoud2013achieving,yang2021multivariable,munoz2023compressing} which present several evidences of dramatic reduction in the number of nodes in a branch-and-bound tree when using dense disjunctions in comparison to branch-and-bound trees based on $1$-sparse disjunctions. Moreover, the papers~\cite{jeroslow1974trivial,chvatal1980hard} present examples of MILPs where every $1$-sparse branching scheme leads to exponential size branch-and-bound trees, although these instances can be solved using polynomial-size branch-and-bound trees when using denser inequalities see~\cite{yang2021multivariable,basu2021complexity}. While the worst-case size of a branch-and-tree may be exponential even when using dense disjunctions~\cite{dadush2020complexity,dey2021lower,dey2022lower,glaser2024sub}, the papers~\cite{pataki2010basis,basu2021complexity} present other compelling theoretical evidence on the importance of branching using dense disjunctions.   

The papers~\cite{owen2001experimental,mahajan2009experiments,mehrotra2011branching,yang2021multivariable} show significant improvement in the size of the branch-and-bound tree by using split disjunctions of a specified sparsity level together which having the coefficients of the associated split sets being in $\{-1, 0, 1\}$. One perspective to view this line of work, is that they explore the paradigm of expanding the list of disjunctions used to build the branch-and-bound tree, from the typically used $1$-sparse disjunctions, to a finite list of pre-specified denser disjunctions. In this paper, we explore a geometric problem motivated by the use of such pre-specified finite lists of dense disjunctions to solve binary MILPs.  
\section{Main results}\label{sec:main}

\subsection{Dominance result for $2$-sparse disjunctions}
Consider two split sets $S(\pi^1, \eta^1)$ and $S(\pi^0, \eta^0)$ in $\mathbb{R}^n$. We say that $S(\pi^1, \eta^1)$ dominates $S(\pi^0, \eta^0)$ if
\begin{eqnarray}\label{eq:dom}
S(\pi^1, \eta^1) \cap [0, 1]^n \supseteq S(\pi^0, \eta^0) \cap [0, 1]^n.
\end{eqnarray}

If (\ref{eq:dom}) holds, then in any branch-and-bound tree that solves a binary MILP using the disjunction corresponding to $S(\pi^0, \eta^0)$, we may replace this disjunction by the disjunction corresponding to $S(\pi^1, \eta^1)$, resulting in a branch-and-bound tree that cannot increase in size in comparison to the original branch-and-bound tree.  

Let $\mathcal{F}_k$ be the finite list of $k$-sparse split sets, such that $S(\pi, \eta) \in \mathcal{F}_k$ if $\pi \in \{-1, 0, 1\}^n$, $\|\pi\|_0 \leq k$, and $\eta \in \{-k, \dots, -1, 0, 1, \dots,  k\}$. If we only use $1$-sparse disjunctions, then clearly there are only $n$ possible split sets in $\mathcal{F}_1$, none of which dominate each other. Next let us  consider the case of $2$-sparse disjunctions. 

\begin{proposition}\label{thm:2}
Consider any arbitrary $2$-sparse split set $S(\pi, \eta)\subseteq \mathbb{R}^n$, that is, $\pi \in \mathbb{Z}^n$ and $\|\pi\|_0 \leq 2$. Then there exists a split set in  $\mathcal{F}_2$ that dominates $S(\pi, \eta)$.
\end{proposition}

Proposition~\ref{thm:2} shows that if one decides to branch using $2$-sparse disjunctions only for solving binary MILPs, there is no reason to use general $2$-sparse split disjunctions -- in particular, one may restrict the use of disjunctions to the finite list described in Proposition~\ref{thm:2}. Indeed, the paper~\cite{yang2021multivariable} shows the importance of branching using $2$-sparse disjunctions by employing exactly the split sets described in Proposition~\ref{thm:2} and shows significant improvement over sizes of tree constructed using the $1$-sparse disjunctions. See Section~\ref{sec:proof2} for a proof of Proposition~\ref{thm:2}.

Generalizing the result of Proposition~\ref{thm:2}, we would like to fix the level of sparsity $k$ of the split disjunctions used to build branch-and-bound tree and ask the question: Does there exist a finite list of $k$-sparse disjunctions, such that it is sufficient to restrict attention to this finite list in order to get the full power of branching with $k$-sparse disjunctions. Unfortunately, as shown in the next result, such finite lists do not exist for $k$-sparse disjunctions with $k\geq 3$.

\begin{theorem}\label{thm:3}
Let $k \geq 3$. There does not exist any finite list $\mathcal{F}$ of $k$-sparse split sets such that any arbitrary $k$-sparse split set is dominated by exactly one of the split sets from $\mathcal{F}$.
\end{theorem}

This negative result in the context of the use of split disjuctions  in branch-and-bound is in striking contrast to the case of cutting planes computed using split sets: for any rational polyhedral there exists a finite list of split sets such that cutting planes derived from an arbitrary split set are dominated by cutting planes derived using one split set from this finite list  \cite{andersen2005split,AVERKOV2012209,dash2017}. See Section~\ref{sec:proof3} for a proof of Theorem~\ref{thm:3}. 

\subsection{Lower bound on covering number for general finite list of dense disjunctions with $k \geq 4$ }
Given the negative result of Theorem~\ref{thm:3}, the next natural question to ask is if there exists finite list of $k$-sparse split sets, such that any other arbitrary $k$-sparse split set is a subset of an union of a small number of split sets from the list. Formally, given split sets $\{S(\pi^i, \eta^i ) \}_{i =0}^p \subseteq \mathbb{R}^n$, we say that $\{S(\pi^i, \eta^i ) \}_{i =1}^p$ dominates $S(\pi^0, \eta^0) $ if:
\begin{eqnarray}\label{eq:uniondom}
\left(\bigcup_{i = 1}^p S(\pi^i, \eta^i)\right) \cap [0, 1]^n \supseteq S(\pi, \eta) \cap [0, 1]^n.
\end{eqnarray}
If (\ref{eq:uniondom}) holds, then in any branch-and-bound tree that solves a binary MILP using the disjunction corresponding to $S(\pi^0, \eta^0)$, we may replace this disjunction by the disjunctions corresponding to $\{S(\pi^i, \eta^i ) \}_{i =0}^p$ resulting in a branch-and-bound tree whose size is no more than $2^{p -1}$ times the original branch-and-bound tree. 
\begin{definition}[Covering number for a finite list of $k$-sparse split sets]
Let $\mathcal{F}$ be a finite list of $k$-sparse split sets. Given an arbitrary $k$-sparse split set $S$, let $\mathcal{F}(S)$ be the smallest number of split sets from $\mathcal{F}$ that dominates $S$. We define the covering number of $\mathcal{F}$, denoted as $C(\mathcal{F})$, as:
$$C(\mathcal{F}):= \textup{max}\{\mathcal{F}(S)\,|\,S\ \textup{ is a }k\textup{-sparse split set}\}.$$ 
\end{definition}
If one can show that a finite list of $k$-sparse disjunctions has a small covering number, then it could be considered a theoretical justification for using just this finite list of pre-specified $k$-sparse disjunctions instead of general $k$-sparse disjunctions. 

The covering number of $\mathcal{F}_1$ is $k$, since, for example, in order to dominate the split set $\{x\in \mathbb{R}^k \,|\, k - 1 < \sum_{i = 1}^kx_i < k\}$ we require all the $k$ disjunctions $0 < x_i < 1$ for $i \in \{1,\ldots,k\}$.
Unfortunately, the next result indicates that it is not possible to find a finite list of disjunctions with significantly smaller covering number.
\begin{theorem}\label{thm:lower}
Let $\mathcal{F}$ be any finite list of $k$-sparse split sets. Then $C(\mathcal{F}) \geq \left\lfloor \frac{k}{2} \right\rfloor$.
\end{theorem}
See Section~\ref{sec:lower} for a proof of Theorem~\ref{thm:lower}.

\subsection{Covering number of $\{-1, 0, 1\}$-disjunctions}
Finally, since a number of papers have successfully employed the very natural list of disjunctions with coefficients only in $\{-1, 0, 1\}$, we explore the covering number of such finite list of disjunctions for sparsity level less or equal than $4$. 

\begin{proposition}\label{thm:upper}
For $k = 2,3,4$ we have that $C(\mathcal{F}_k) \leq k -1$. 
\end{proposition}

See Section~\ref{sec:upper} for a proof of Proposition~\ref{thm:upper}.
\section{Conclusions}
The results of this paper justify the use of pre-specified list of disjunctions with coefficients in $\{-1, 0, 1\}$ for low levels of sparsity. For $k=2$, Proposition~\ref{thm:2} provides this justification. For a branch-and-bound tree using $3$-sparse disjunctions, Theorem~\ref{thm:lower} and Proposition~\ref{thm:upper} imply that any finite list has a covering number of at least $2$ and 
$\mathcal{F}_3$ also has a covering number of $2$. Thus with respect to covering number, it is optimal to limit the use of disjunctions from $\mathcal{F}_3$.
It is an open question if 
$\mathcal{F}_k$ is optimal for higher values of $k$ with respect to covering number. In order to answer this question, results of both Theorem~\ref{thm:lower} and Proposition~\ref{thm:upper} may need to be tightened and generalized.  

More generally, Theorem~\ref{thm:lower} may also be an indication that the use of pre-specified list of disjunctions may not be the best way to generate small branch-and-bound trees. While using disjunctions in $\mathcal{F}_k$ 
already produces smaller branch-and-bound trees than those produced using $1$-sparse  disjunctions~\cite{owen2001experimental,mehrotra2011branching,fukasawa2020split,yang2021multivariable}, in order to truly obtain significantly smaller branch-and-bound trees, one may need to further develop and expand on methods to select problem-specific dense disjunctions that are not pre-specified~\cite{aardal2004hard,mahajan2009experiments,karamanov2011branching,cornuejols2011improved,mahmoud2013achieving,yang2021multivariable,munoz2023compressing}.

\section{Proof of Proposition~\ref{thm:2}} \label{sec:proof2}

In order to prove Proposition~\ref{thm:2} ($k=2$) and Proposition~\ref{thm:upper} ($k=3,4$) in Section \ref{sec:upper}, we have to show that for any given arbitrary split set $S = \{ x \in \mathbb{R}^k\,|\, \eta < \pi^{\top}x < \eta +1 \}$ at most $k-1$ split sets from $\mathcal{F}_k$ are needed to dominate it. Without loss of generality, we may assume that $0 \leq \pi_1\leq \pi_2 \leq \ldots \leq \pi_k$. This is because, if $\pi_i < 0$ we can change $x_i$ to $1-x_i$, and then permute the order of the variables. Note that this is fine because $\mathcal{F}_k$ is closed under taking the same operations. 

\begin{proof}[Proof for $k =2$] We assume $\|\pi\|_0 = 2$, since otherwise the result is trivial. 
Let $x\in S\cap[0,1]^2$. We consider the following cases.

\begin{itemize}
\item $0\leq\eta<\eta+1\leq \pi_1$: 
For $x \in S$, we have  $x_1+x_2\leq x_1+\frac{\pi_2}{\pi_1}x_2  = \frac{\pi_1 x_1 + \pi_2 x_2}{\pi_1}<\frac{\eta+1}{\pi_1}\leq 1.$ Since $(0,0) \notin S$, we obtain $S \cap [0,1]^2 \subseteq \{x \in [0, 1]^2 \,|\,0<x_1+x_2<1\}$.
	
\item $0 \leq \pi_1\leq \eta;\ \eta+1\leq\pi_2$: We have $0\leq\frac{\eta-\pi_1}{\pi_2}\leq \frac{\eta-\pi_1x_1}{\pi_2}<x_2$ for $x\in S$. On the other hand, for $x \in S$ we have $x_2\leq \frac{\pi_1}{\pi_2}x_1+x_2< \frac{\eta+1}{\pi_2}\leq 1$. Thus, $S \cap [0,1]^2 \subseteq \{x \in [0, 1]^2 \,|\, 0<x_2<1\}$.
		
\item  $0 \leq \pi_1\leq \eta;\ 0 \leq \pi_2 \leq \eta$: We have $1<\frac{\pi_1}{\eta}x_1+\frac{\pi_2}{\eta}x_2\leq x_1 + x_2$ for $x\in S$. Since $(1,1)\notin S$, we obtain $S \cap [0,1]^2 \subseteq \{x \in [0, 1]^2 \,|\ 1< x_1 + x_2<2\}$.
\end{itemize}
\end{proof}
\section{Proof of Theorem~\ref{thm:3}} \label{sec:proof3}
We will prove Theorem~\ref{thm:3} for $k =3$. A similar proof can be given for $k\geq 4$, but the result in this case is implied by Theorem~\ref{thm:lower} so we do not consider it in this section. 
\begin{proof}[Proof of Theorem~\ref{thm:3}]
In order to prove Theorem~\ref{thm:3}, we show that for the infinite family of split  sets
$$S_\gamma=\{x\in\mathbb{R}^3\, |\, \gamma<x_1+\gamma x_2+(\gamma+1)x_3<\gamma+1\},$$
\noindent where $\gamma\in\mathbb{Z}_+, \gamma\geq 1$, there is no split set in $\mathbb{R}^3$ that  contains $S_\gamma \cap [0, 1]^3$ for infinitely many values of $\gamma$. Assume for a contradiction that  there exists an split set $S=\{x\in\mathbb{R}^3 \, |\,  \eta<\pi^Tx<\eta+1\}$,
where $\pi \in(\pi_1, \pi_2, \pi_3) \in\mathbb{Z}^3,\eta\in\mathbb{Z}$, such that $S$ dominates $S_\gamma$ for infinitely many $\gamma\in\mathbb{Z}_+$, that is, 
\begin{equation}\label{eq:inclusion}
S_\gamma\cap[0,1]^3\subseteq S\cap[0,1]^3 \ \forall \gamma \in \Gamma \quad (\Rightarrow \overline{S}_\gamma\cap[0,1]^3\subseteq \overline{S}\cap[0,1]^3, \ \forall \gamma \in \Gamma),
\end{equation}
where $\overline{S}_\gamma$ and $\overline{S}$ are closure of ${S}_\gamma$ and ${S}$ respectively, and $\Gamma \subseteq \mathbb{Z}_+$ is an infinite set.

We first show that:
\begin{align}\label{eq:rev_inclusion_L}
H^0:=\{x\in [0,1]^3\,|\,\pi^Tx\leq \eta\}&\subseteq \{x\in [0,1]^3\,|\,x_1+\gamma x_2+(\gamma+1)x_3\leq \gamma\}=:H_\gamma^0,\\
\label{eq:rev_inclusion_R}
H^1:=\{x\in [0,1]^3\,|\,\pi^Tx\geq \eta+1\}&\subseteq \{x\in [0,1]^3\,|\,x_1+\gamma x_2+(\gamma+1)x_3\geq \gamma+1\}=:H_\gamma^1.
\end{align}

Notice that we must have $H_\gamma^j\cap\{0,1\}^3\neq\emptyset$ for all $j\in \{0,1\}$; otherwise if, for instance $H_\gamma^0\cap\{0,1\}^3=\emptyset$, then we would have $\{0,1\}^3\subseteq H_\gamma^1$ which implies $[0,1]^3\subseteq H_\gamma^1$, a contradiction with the fact $S_\gamma\cap [0,1]^3 \neq \emptyset$. On the other hand, observe that $S_\gamma$ being dominated by $S$ is equivalent to: for all $i\in \{0,1\}$ there exists $j\in \{0,1\}$ such that $H^i\subseteq H_\gamma^j$. Since $(H^0\cup H^1)\cap\{0,1\}^3=\{0,1\}^3$, we conclude that it cannot happen that $H^0\subseteq H_\gamma^j$ and $H^1\subseteq H_\gamma^j$ for the same $j$ since $H_\gamma^i\cap\{0,1\}^3\neq\emptyset$ for $i\neq j$. Therefore, \eqref{eq:rev_inclusion_L} and \eqref{eq:rev_inclusion_R} hold (we may assume that  we have $H^i\subseteq H_\gamma^i$ for $i=\in \{0,1\}$ by considering $S$ to be defined by $\hat\pi=-\pi$ and $\hat\eta=-\eta-1$ if necessary).

Since $(0,1,0)$ satisfies the equation $x_1+\gamma x_2+(\gamma+1)x_3= \gamma$ and by \eqref{eq:rev_inclusion_L} we have $H^0\cap\{0,1\}^3=H_\gamma^0\cap\{0,1\}^3$, we must have that $(0,1,0)$ satisfies the inequality $\pi^Tx\leq \eta$. We now show that $(0,1,0)$ must satisfy $\pi^Tx= \eta$. Assume for a contradiction that it satisfies $\pi^Tx< \eta$. Let $x_0 \in S_\gamma\cap [0,1]^3$ be an arbitrary point. For $\lambda>0$ small enough we have that the point $x_\lambda=(0,1,0)+\lambda (x_0-(0,1,0))$ satisfies $\pi^Tx_\lambda< \eta$ and, by convexity of $\overline{S}_\gamma\cap[0,1]^3$, that $x_\lambda\in S_\gamma\cap [0,1]^3$.
Since $S_\gamma\cap [0,1]^3\subseteq S\cap [0,1]^3$, it follows that $x_\lambda\in S\cap [0,1]^3$, a contradiction with the fact that  $\pi^Tx_\lambda< \eta$. Thus, we must have that 
$\pi_2=\eta$. By a similar argument, since $(1,1,0)$ and $(0,0,1)$ satisfy $x_1+\gamma x_2+(\gamma+1)x_3= \gamma+1$, it follows from \eqref{eq:rev_inclusion_R} that we must have that these points satisfy $\pi^Tx= \eta+1$, and therefore $\pi_1+\pi_2=\eta+1$ and $\pi_3=\eta+1$. Therefore, we obtain that $\pi_1=1$, $\pi_2=\eta$ and $\pi_3=\eta+1$.

Since $(1,0,\gamma/(\gamma+1))\in\overline{S}_\gamma\cap[0,1]^3$, by \eqref{eq:inclusion} we obtain
$1 + (\eta+1)\frac{\gamma}{\gamma+1}\leq \eta+1 \Leftrightarrow 1 \leq \frac{\eta +1}{\gamma +1}.$ Since this inequality holds for any $\gamma\in\Gamma$, we obtain $1\leq 0$, a contradiction.

\end{proof}

\section{Proof of Theorem~\ref{thm:lower}}\label{sec:lower}
We first prove the result when the sparsity level is an even positive integer $2k$. 

Consider the following family of split sets parameterized by a positive integer $\theta$:
\begin{align*}
    \begin{array}{rl}
    & S^{\theta} = \left\{ (x,y) \in \mathbb{R}^k \times \mathbb{R}^k\,\left|\, \sum\limits_{i=1}^k \theta^i < \sum\limits_{i=1}^k \theta^i (x_i + y_i) < 1+ \sum\limits_{i=1}^k \theta^i \right.\right\}.
    \end{array}
\end{align*}

In order to prove Theorem~\ref{thm:lower} it is sufficient to prove the following result:
\begin{lemma}   \label{lem_result}
    For every finite collection of split sets $\mathcal{F}$, there exists $\theta \in \mathbb{Z}_+$ with $\theta \geq 1$, such that one needs at least an union of $k$ split  sets from $\mathcal{F}$ to dominate $S^{\theta} \cap [0,1]^{2k}$. 
\end{lemma}

Before presenting the proof of Lemma~\ref{lem_result}, we introduce some notation. Consider a list of split sets
 $A^{(i)} := \left\{ (x,y) \,|\, c^{(i)} < a^{(i)} x + b^{(i )} y^{(i)} < c^{(i)} + 1\right\},\ \textup{for}\ i=1,\ldots,p.$
For each $i$, we denote the two connected components of the complement set to $A^{(i)}$  by
$$A^{(i)}_0 :=  \left\{ (x,y) \,|\, a^{(i)} x + b^{(i )} y^{(i)} \leq c^{(i)}\right\}\quad \textup{and}\quad A^{(i)}_1 :=  \left\{ (x,y) \,|\, a^{(i)} x + b^{(i )} y^{(i)} \geq c^{(i)} + 1\right\}.$$

Given a binary vector $u \in \{0,1\}^p$, we further define $A_{u} = \bigcap_{i=1}^p A^{(i)}_{u_i}$.
Note that the fact that $\bigcup_{i=1}^p A^{(i)}$ dominates $S^{\theta} $ can be written as:
$$S^{\theta} \cap [0,1]^{2k} \subseteq \left(\bigcup_{i \in [g]} A^{(i)}\right) \cap [0,1]^{2k} \
    \Leftrightarrow \  [0,1]^{2k} \setminus S^{\theta} \supseteq [0,1]^{2k} \setminus \left(\bigcup_{i \in [g]} A^{(i)}\right).$$
So dominance of the given list of split sets is equivalent to:
\begin{equation}\label{key_relationship}
    \forall u \in \{0,1\}^p, \text{ either }  A_{u} \cap [0,1]^{2k} \subseteq S^{\theta}_0 \cap [0,1]^{2k} \text{ or } A_{u} \cap [0,1]^{2k} \subseteq S^{\theta}_1 \cap [0,1]^{2k}. 
\end{equation}

Now we present a proof of Lemma \ref{lem_result}.
\begin{proof} 
We argue by contradiction. Suppose one needs at most $k -1$ split sets from $\mathcal{F}$ to dominate $S^{\theta}$  for all $\theta$.
Since $\mathcal{F}$ is finite, but there are infinitely many choices of $S^\theta$, 
there must exist  $p$ split sets from $\mathcal{F}$, where $p\leq k -1$, and an infinite set $\Theta \subseteq \mathbb{Z}_+$ such that those $p$ split  sets dominate $S^\theta \cap [0,1]^{2k}$ for all $\theta \in \Theta$. We denote those split sets by 
\begin{align*}
    A^{(i)} := \left\{ (x,y) \,|\, c^{(i)} < a^{(i)} x + b^{(i )} y^{(i)} < c^{(i)} + 1\right\},\ \textup{for}\ i=1,\ldots,p.
\end{align*} 

We will show that  $(\ref{key_relationship})$ fails for sufficiently large $\theta\in \Theta$. Our main idea is to construct a certain point $z \in [0,1]^{2k}$ such that $z \in A_{u} \cap [0,1]^{2k}$ for some $u$ but $z$ violates $(\ref{key_relationship})$. 

Consider the following linear system:
\begin{align}
     \label{eq1a} a^{(i)} x + b^{(i )} y = 0\ &\quad\textup{for}\ i=1,\ldots,p \\
     \label{eq1b} y_i = 0\ &\quad\textup{for}\ i=1,\ldots,k.
\end{align}

This linear system has $2k$ variables and $k+p$ constraints. Since $k + p< 2k$ it has at least one non-zero solution $(x^*,y^*)$. Without loss of generality, we may assume that $\|{(x^*,y^*)}\|_2 = 1$ and $x^*_j > 0$ where $j$ is the largest index $i=1,\ldots, k$ such that $x^*_i \neq 0$. 

By \eqref{eq1b} and for sufficiently large $\theta\in \Theta$ we have that
\begin{align}
    \label{eq2} \sum_{i =1}^k \theta^i (x_i^* + y_i^* ) = \sum_{i =1}^j \theta^i x_i^* > 0.
\end{align}

We now construct a binary vector $(s, t) \in \{0,1\}^{2k}$ in the following way:
$$ s_i = 0,  t_i = 1 \text{ if } x^*_i \geq 0\quad\textup{and}\quad  s_i = 1, t_i = 0 \text{ if }  x^*_i < 0.$$


Notice that  since $(s, t)$ is a integer vector, it must belong to either $A^{(i)}_0$ or $A^{(i)}_1$ for all $i=1,\ldots,p$ and therefore $(s, t) \in A_{u^*} \cap [0,1]^{2k}$ for some $u^*$. 

We now verify that $(s, t) + \lambda (x^*,y^*) \in A_{u^*} \cap [0,1]^{2k}$ for some sufficiently small $\lambda > 0$. Indeed, $(s, t) + \lambda (x^*,y^*)$ stays in $A_{u^*}$ for any $\lambda>0$ because of \eqref{eq1a}. On the other hand, $(s, t) + \lambda (x^*,y^*)$ stays in $[0,1]^{2k}$ for sufficiently small $\lambda > 0$ because $t_i + \lambda y^*_i$ does not change due to \eqref{eq1b}, components associated to $s_i=1$ decrease a little and components associated to $s_i=0$ increase a little.

Now observe that $\sum\limits_{i \in [k]} \theta^i (s_i + t_i ) = \sum\limits_{i \in [k]} \theta^i$ and $\sum\limits_{i \in [n]} \theta^i (x_i^* + y_i^* ) > 0$ by (\ref{eq2}), hence we obtain that
\begin{align*}
    \sum\limits_{i =1}^k \theta^i < \sum\limits_{i =1}^k \theta^i (s_i + \lambda x_i^*   + t_i  + \lambda y_i^*)  < 1 + \sum\limits_{i =1}^k \theta^i,
\end{align*}
for sufficiently small $\lambda >0$. 
In other words, for sufficiently small $\lambda>0$ and large enough $\theta\in \Theta$  we have that $(s, t) + \lambda (x^*,y^*) \in S^{\theta}\cap ( A_{u^*} \cap [0,1]^{2k})$. We conclude that $(\ref{key_relationship})$ is not satisfied for the point $(s, t) + \lambda (x^*,y^*)$, a contradiction.
\end{proof}

In order to prove Theorem~\ref{thm:lower} for odd sparsity levels of split disjunctions, a similar proof can be presented using the family of split sets:
\begin{align*}
    \begin{array}{rl}
    & S^{\theta} := \left\{ (x,y) \in \mathbb{R}^k \times \mathbb{R}^{k+ 1}: \sum\limits_{i=1}^k \theta^i < \sum\limits_{i=1}^k \theta^i (x_i + y_i) + y_{k +1} < 1+ \sum\limits_{i=1}^k \theta^i\right\}.
    \end{array}
\end{align*}

\section{Proof of Proposition~\ref{thm:upper}}\label{sec:upper}
The case $k=2$ is proven in Proposition~\ref{thm:2}. We now consider the cases $k =3,4$. 


\begin{proof}[Proof for $k =3$] Let  $S= \{ x \in \mathbb{R}^3\,|\, \eta < \pi^{\top}x < \eta +1 \}$, recall that we may assume that $0 \leq \pi_1\leq \pi_2 \leq \pi_3$ (see Section \ref{sec:proof2}). 
We have to show that at most $2$ split sets from $\mathcal{F}_3$ are needed to dominate it. There are three cases:
\begin{itemize}
\item $\pi_3 \geq \eta + 1$: In this case, observe that $x \in S$, implies that $x_3 < 1$. By Proposition~\ref{thm:2}, we know there exist $1$ split set of sparsity $2$ (or lesser) from $\mathcal{F}_3$ whose union contains the set $\{x \in [0, 1]^3\,|\, x \in S, x_3 = 0\}$. The set of points in $\{ x \in [0, 1]^3 \,|\, x \in S, 0 < x_3 < 1\}$ is contained in the split set $0 < x_3 < 1$.
\item $\pi_1 + \pi_2 \leq \eta$: In this case, observe that $x \in S$, implies that $x_3 > 0$. By Proposition~\ref{thm:2}, we know there exist $1$ split set of sparsity $2$ (or lesser) from $\mathcal{F}_3$ whose union contains the set $\{x \in [0, 1]^3\,|\, x \in S, x_3 = 1\}$. The set of points in $\{ x \in [0, 1]^3 \,|\, x \in S, 0 < x_3 < 1\}$ is contained in the split set $0 < x_3 < 1$.

\item $\pi_3 \leq \eta$ and $\pi_1 + \pi_2 \geq \eta +1$: 
Since $\pi_1\leq \pi_2 \leq \pi_3 \leq \eta$, if $\sum_{j = 1}^3 x_j \leq 1$, then $\sum_{j = 1}^3 \pi_jx_j \leq \eta$. Thus, $x \in S$ implies that $\sum_{j = 1}^3 x_j > 1$. Moreover, if $\sum_{j = 1}^3 x_j \geq 2$, then $\sum_{j =1}^3\pi_jx_j \geq \pi_1 + \pi_2 \geq \eta + 1$. Thus $x \in S$ implies that $\sum_{j = 1}^3 x_j < 2$. Therefore, $S$ is dominated by the split set $\{x \in \mathbb{R}^3\,|\, 1 < \sum_{j =1}^3 x_j <2 \}$. 
\end{itemize}
\end{proof}


\begin{proof}[Proof for $k =4$] Let $S= \{ x \in \mathbb{R}^4\,|\, \eta < \pi^{\top}x < \eta +1 \}$ with $0 \leq \pi_1\leq \pi_2 \leq \pi_3\leq \pi_4$. 
There are ten cases:
\begin{itemize}
\item $\pi_4 \geq \eta + 1$: In this case, observe that $x \in S$, implies that $x_4 < 1$. By Proposition~\ref{thm:upper} for $k =3$ case, we know there at most $2$ split set of sparsity $3$ (or lesser) from $\mathcal{F}_4$ whose union contains the set $\{x \in [0, 1]^4\,|\, x \in S, x_4 = 0\}$. The set of points in $\{ x \in [0, 1]^4 \,|\, x \in S, 0 < x_4 < 1\}$ is contained in the split set $0 < x_4 < 1$.

\item $\pi_1 + \pi_2  + \pi_3 \leq \eta$: In this case, observe that $x \in S$, implies that $x_4 > 0$. By Proposition~\ref{thm:upper} for $k = 3$, we know there at most $2$ split set of sparsity $3$ (or lesser) from $\mathcal{F}_4$ whose union contains the set $\{x \in [0, 1]^4\,|\, x \in S, x_4 = 1\}$. The set of points in $\{ x \in [0, 1]^4 \,|\, x \in S, 0 < x_4 < 1\}$ is contained in the split set $0 < x_4 < 1$.

\item $\pi_1 + \pi_2 \geq \eta +1$: We may assume that $\pi_4 \leq \eta$. Thus, we have $x\in S$ implies $1 < x_1 + x_2 + x_3 + x_4$. On the other hand, we also must have $x_1 + x_2 + x_3 + x_4 <2$, since otherwise, $\sum_{j =1}^4\pi_jx_j \geq \pi_1 + \pi_2 \geq \eta +1$. Thus $S \cap [0, 1]^4$ is contained in $1 < x_1+ x_2 + x_2 + x_3 + x_4 <2$.

\item $\pi_1 + \pi_3 \geq \eta + 1$: We may assume $\pi_4 \leq \eta$ and $\pi_1 + \pi_2 \leq \eta$. We claim that $S$ is contained in the union of $1 < x_1 + x_2 + x_3 + x_4 < 2$ and $0< x_3 + x_4 < 1$. 
Consider the following cases for $x \in S$:
\begin{itemize}
\item If $x_1 + x_2 \leq 1$: We claim that that $x_1 + x_2 + x_3 + x_4 < 2$. Assume by contradiction $x_1 + x_2 + x_3 + x_4 \geq 2$. Then we have $x_3 + x_4 \geq 1$ and thus 
$\sum_{j =1}^4\pi_jx_j \geq \pi_1 \cdot \textup{min}\{1, 2 - x_3  - x_4\} + \pi_3 \cdot\textup{max}\{1, x_3 + x_4\} \geq \pi_1 + \pi_3 \geq \eta +1.$
On the other hand, since $\pi_4 \leq \eta$, we have $1 < x_1 + x_2 + x_3 + x_4$. Thus, in this case $x$ belongs to $1 < x_1+ x_2 + x_2 + x_3 + x_4 <2$.
\item If $x_1 + x_2 > 1$: Then note that $x_3 + x_4 < 1$, since otherwise $\sum_{j =1}^4\pi_jx_j \geq \pi_1  + \pi_3 \geq \eta +1$. Also note that if $x_3 + x_4 = 0$, then $\sum_{j =1}^4\pi_jx_j \leq \pi_1 + \pi_2 \leq \eta$. Thus, in this case we have that $x$ belongs $0 <  x_3 + x_4 < 1$.
\end{itemize}

\item $\pi_3 + \pi_4 \leq \eta$: We assume that $\pi_1+ \pi_2 + \pi_3 \geq \eta +1$. First, note that $S$ is contained in $x_1 + x_2 + x_3 + x_4 > 2$. Also, since $\pi_1+ \pi_2 + \pi_3 \geq \eta +1$, we have that $S$ is contained in $x_1 + x_2 + x_3 + x_4 < 3$. Thus, $S$ is contained in $2 < x_1 + x_2 + x_3 + x_4 < 3$.

\item $\pi_2 + \pi_4 \leq \eta$: We may assume $\pi_1+ \pi_2 + \pi_3 \geq \eta +1$ and $\pi_3 + \pi_4 \geq \eta +1$. We claim that $S$ is contained in the union of $2 < x_1 + x_2 + x_3 + x_4 < 3$ and $1< x_3 + x_4 < 2$. 
Consider the following cases for $x \in S$:
\begin{itemize}
\item If $x_3 + x_4 \leq 1$: We claim that that $x_1 + x_2 + x_3 + x_4 > 2$. Assume by contradiction $x_1 + x_2 + x_3 + x_4 \leq 2$. Thus, 
$\sum_{j =1}^4\pi_jx_j \leq \pi_2 \cdot \textup{max}\{1, 2 - x_3  - x_4\} + \pi_4 \cdot\textup{min}\{1, x_3 + x_4\} \leq \pi_2 + \pi_4 \leq \eta.$
On the other hand, since $\pi_1+ \pi_2 + \pi_3  \geq \eta +1$, we have $3 > x_1 + x_2 + x_3 + x_4$. Thus, in this case $x$ belongs to $2 < x_1+ x_2 + x_2 + x_3 + x_4 < 3$.
\item Also note that  $x_3 + x_4 = 2$ is not possible, since then $\sum_j \pi_j x_j \geq \pi_3 + \pi_4 \geq \eta +1$.
\end{itemize}
Thus, $S \cap [0, 1]^4$ is contained in the union of $2 < x_1 + x_2 + x_3 + x_4 < 3$ and $1< x_3 + x_4 < 2$. 

\item $\pi_2 + \pi_3 \geq \eta + 1$ and $\pi_1 + \pi_4 \geq \eta +1$: We may assume $\pi_4 \leq \eta$ and $\pi_1 + \pi_3 \leq \eta$. We claim that $S$ is contained in the union of $1 < x_2 + x_3 + x_4 < 2$, $0<x_1 <1$ and $0< x_4 < 1$. 
Consider the following cases:
\begin{itemize}
\item $x_1 = 0$: In this case, note that because of $\pi_2 + \pi_3 \geq \eta + 1$ and $\pi_4 \leq \eta$, we have that $x$ belongs to $1 < x_2 + x_3 + x_4 < 2$.
\item $0<x_1 < 1$: In this case, note that $x$ belongs to $0 < x_1  < 1$.
\item $x_1 = 1$: Clearly, due to $\pi_2 + \pi_3 \geq \eta + 1$ we have that $ x_2 + x_3 + x_4 < 2$. If $1 < x_2 + x_3 + x_4$, then $x$ belongs to $1 < x_2 + x_3 + x_4 < 2$.

Otherwise suppose, $x_2 + x_3 + x_4 \leq 1$. We claim that $0 < x_4 <1$. By contradiction, if $x_4 = 0$, then note that $\sum_j \pi_jx_j \leq \pi_1 + \pi_3 \leq \eta$. If $x_4 =1$, then note that $\sum_j \pi_jx_j \geq \pi_1 + \pi_4 \geq \eta + 1.$
\end{itemize}

\item $\pi_2 + \pi_3 \leq \eta $ and $\pi_1 + \pi_4 \leq \eta$: We may assume $\pi_1+ \pi_2 + \pi_3  \geq \eta + 1$ and $\pi_2 + \pi_4 \geq \eta + 1$. We claim that $S$ is contained in the union of $0 < x_2 + x_3 < 1$, $2< x_1 + x_2 + x_3 < 3$ and $0< x_4 < 1$. 
Consider the following cases:
\begin{itemize}
\item $x_4 = 0$:  In case, note that due to $\pi_1+ \pi_2 + \pi_3  \geq \eta + 1$, we have that $x_1 + x_2 + x_3 < 3$. Also, since $\pi_2 + \pi_3 \leq \eta$, we have that $x_1 + x_2 + x_3 > 2$. Thus, $x$ is contained in $2 < x_1 + x_2 + x_3 < 3$. 
\item $0< x_4 < 1$: In this case, note that $x$ belongs to $0 < x_4  < 1$.
\item $x_4 = 1$: In this case note that $x_2 + x_3 \geq 1$ is not possible, since $\sum_j\pi_jx_j \geq \pi_2 +  \pi_4 \geq \eta +1$. Also note that $x_2 + x_3 = 0$ is not possible, since that $\sum_j\pi_jx_j \leq \pi_1 + \pi_4 \leq \eta$. Thus, in this case, $x$ belongs to $0 < x_2 + x_3 <1$.
\end{itemize}

\item $\pi_2 + \pi_3 \geq \eta + 1$ and $\pi_1 + \pi_4 \leq \eta$: We may assume $\pi_1+ \pi_2 + \pi_3  \geq \eta + 1$, $\pi_1 + \pi_3 \leq \eta$ and $\pi_4 \leq \eta$. We claim that $S$ is contained in the union of $1 < x_2 + x_3 + x_4 < 2$, $0< x_1 < 1$, and $0<x_2  + x_3<1$. Consider the following cases:
\begin{itemize}
\item $x_1 = 0$:  In this case, note that because of $\pi_2 + \pi_3 \geq \eta + 1$ and $\pi_4 \leq \eta$, we have that $x$ belongs to $1 < x_2 + x_3 + x_4 < 2$.
\item $0< x_1 < 1$: In this case, note that $x$ belongs to $0 < x_1  < 1$.
\item $x_1 = 1$: Clearly, due to $\pi_2 + \pi_3 \geq \eta + 1$ we have that $ x_2 + x_3 + x_4 < 2$. If $1 < x_2 + x_3 + x_4$, then $x$ belongs to $1 < x_2 + x_3 + x_4 < 2$.

Otherwise suppose, $x_2 + x_3 + x_4 \leq 1$. In this case, note that $x_2 + x_3 = 1$ is not possible, since that $x_4 = 0$ and we have $\sum_j\pi_jx_j \leq \pi_1 +  \pi_3 \leq \eta$. Also note that $x_2 + x_3 = 0$ is not possible, since that $\sum_j\pi_jx_j \leq \pi_1 + \pi_4 \leq \eta$. Thus, in this case, $x$ belongs to $0 < x_2 + x_3 <1$.
\end{itemize}
\item $\pi_2 + \pi_3 \leq \eta $ and $\pi_1 + \pi_4 \geq \eta + 1$: We may assume $\pi_1+ \pi_2 + \pi_3  \geq \eta + 1$ and $\pi_4 \leq \eta$. We claim that $S$ is contained in the union of $1 < x_1 + x_2 + x_3  + x_4 < 2$, $2< x_1 + x_2 + x_3 < 3$ and $0< x_4 < 1$. 
Consider the following cases:
\begin{itemize}
\item $x_4 = 0$: In case, note that due to $\pi_1+ \pi_2 + \pi_3  \geq \eta + 1$, we have that $x_1 + x_2 + x_3 < 3$. Also, since $\pi_2 + \pi_3 \leq \eta$, we have that $x_1 + x_2 + x_3 > 2$. Thus, $x$ is contained in $2 < x_1 + x_2 + x_3 < 3$.
\item $0< x_4 < 1$: In this case, note that $x$ belongs to $0 < x_4  < 1$.
\item $x_4 = 1$: In this case note that since $\pi_1 + \pi_4 \geq \eta +1$, we have that $x_1 + x_2 + x_3 + x_4 < 2$. Also note that since $\pi_4 \leq \eta$, we have that $x_1 + x_2 + x_3 + x_4 > 1$. Thus, in this case, $x$ belongs to $1 < x_1 + x_2 + x_3 + x_4 < 2$.
\end{itemize}
\end{itemize}
\end{proof}

\section*{Acknowledgements}
We would like to thank Diego Cifuentes, Amitabh Basu, Antoine Deza, and Lionel Pournin for various discussions. We would also like to thank the support from AFOSR grant \# F9550-22-1-0052 and from the  ANID grant  Fondecyt \# 1210348 .
\bibliographystyle{plain}
\bibliography{ref}

\end{document}